\documentclass[12pt]{amsart}
\usepackage[left=3cm, right=3cm, top=2cm]{geometry}
\usepackage{amsmath, amsthm, amssymb, mathtools}
\usepackage{hyperref}
\usepackage{fullpage}
\usepackage{cleveref}
\usepackage{tkz-berge}
\usetikzlibrary{intersections}
\tikzset{help lines/.style={dashed, thick}}
\usepackage[T1]{fontenc}
\usepackage[ansinew]{inputenc}
\usepackage{float}
\usepackage{setspace}
\usepackage{xpatch}
\usepackage{subcaption}
\usepackage{pst-node}
\usepackage{tikz-cd}
\usetikzlibrary{positioning}
\usetikzlibrary{shapes,arrows}
\usepackage{tikz}
\usepackage{graphicx,epsfig}
\usepackage{mathrsfs}
\usepackage{verbatim}

\newtheorem{theorem}{Theorem}[section]
\newtheorem{lemma}[theorem]{Lemma}
\newtheorem{corollary}[theorem]{Corollary}
\newtheorem{definition}{Definition}
\newtheorem{conjecture}{Conjecture}

\newtheorem{proposition}[theorem]{Proposition}

\def\S{\mathbb{S}}

\newcommand{\si}{\sigma}

\numberwithin{equation}{section}

\title{Bounded degree complexes of forests}
\author[Anurag\, Singh]{Anurag Singh}
\address{Department of Mathematics, Chennai Mathematical Institute, India 603103}
\email{asinghiitg@gmail.com}
\date{\today}

\begin{document}

\begin{abstract}
Given an arbitrary sequence of non-negative integers $\vec{\lambda}=(\lambda_1,\dots,\lambda_n)$ and a graph $G$ with vertex set $\{v_1,\dots,v_n\}$, the bounded degree complex, denoted $\text{BD}^{\vec{\lambda}}(G)$, is a simplicial complex whose faces are the subsets $H\subseteq E(G)$ such that for each $i \in \{1,\dots,n\}$, the degree of vertex $v_i$ in the induced subgraph $G[H]$ is at most $\lambda_i$. When $\lambda_i=k$ for all $i$, the bounded degree complex $\text{BD}^{\vec{\lambda}}(G)$ is called the $k$-matching complex, denoted $M_k(G)$.

In this article, we determine the homotopy type of bounded degree complexes of forests. In particular, we show that, for all $k\geq 1$, the $k$-matching complexes of caterpillar graphs are either contractible or homotopy equivalent to a wedge of spheres, thereby proving a conjecture of Julianne Vega \cite[Conjecture 7.3]{Vega19}. We also give a closed form formula for the homotopy type of the bounded degree complexes of those caterpillar graphs in which every non-leaf vertex is adjacent to at least one leaf vertex.

\end{abstract}

\maketitle

%%%%%%%%%%%%%%%%%%%%%%%%=========================

\section{{Introduction}}
Let $G$ be a simple graph without any isolated vertex and $V(G)=\{v_1,\dots,v_n\}$ be the vertex set of $G$. Let $\vec{\lambda}=(\lambda_1,\dots,\lambda_n)$ be a sequence of non-negative integers. The {\itshape bounded degree complex}, denoted $\text{BD}^{\vec{\lambda}}(G)$, is a simplicial complex whose vertices are the edges of $G$ and faces are subsets $H\subseteq E(G)$ such that for each $i \in \{1,\dots,n\}$, the degree of vertex $v_i$ in the induced subgraph $G[H]$ is at most $\lambda_i$ (see \Cref{fig:example of BD} for example). When $\lambda_i =k$ for all $i \in \{1,\dots,n\}$, the bounded degree complex $\text{BD}^{(k,\dots,k)}(G)$ is called the {\itshape $k$-matching complex} of graph $G$ and denoted by $M_k(G)$.

The matching complexes were first introduced by Bouc \cite{Bou92} in connection with Brown complexes and Quillen complexes. His main aim was to understand the homology groups of $M_1(K_n)$ as $S_n$ representations, where $K_n$ is the complete graph on $n$ vertices and $S_n$ is the symmetric group on $n$ element set. He gave a combinatorial formula for the Euler characteristic of $M_1(K_n)$. Prior to Bouc's work on $M_1(K_n)$, Garst \cite{Gar79}, dealing with
Tits coset complexes, introduced the chessboard complex, now known as $1$-matching complex of complete bipartite graphs $K_{m,n}$. He showed that $M_1(K_{m,n})$ is Cohen-Macaulay if and only if $2m-1 \leq n$. For detailed survey of work done on these complexes, we refer the reader to Section $2$ of Wach's survey \cite{Wach03}.

The study of matching complexes have become very popular in the topological combinatorics community because of their applications and the fact that their topological properties are interesting in their own right. For example, for trees \cite{MT08}, paths and cycle graphs \cite{Koz07}, $M_1(G)$ has the homotopy type of wedge of spheres, and in some cases, like $K_{m,n}$ for $2m-1\leq n$, $M_1(G)$ is shellable \cite{Zieg94}. However, this is not the case in general, for example Shareshian and Wachs \cite{SW07} showed that the integral homology groups of matching complexes are not torsion-free in many cases. Therefore it is extremely difficult to determine their homotopy type in general.

 In \cite{MT08}, Marietti and Testa studied the homotopy type of $1$-matching complexes of forests and proved the following.

\begin{theorem}\cite[Theorem 4.13]{MT08}
Let $G$ be a forest. Then $M_1(G)$ is either contractible or homotopy equivalent to a wedge of spheres.
\end{theorem}

Milutinovi{\'c} et al. \cite{MJMV19} determined the closed form formula for the homotopy type of $1$-matching complexes for a special class of trees, known as caterpillar graphs (see \Cref{def:caterpillar graph}). In \cite{Vega19}, Vega studied the homotopy type of $2$-matching complexes of perfect caterpillar graphs (see \Cref{def:perfect caterpillar graphs}) and conjectured the following.
\begin{conjecture}[{\cite[Conjecture 7.3]{Vega19}}]\label{conj:vegas conjecture}
The $k$-matching complex of caterpillar graphs are either contractible or homotopy equivalent to a wedge of spheres.
\end{conjecture}

Now, we turn our attention to bounded degree complexes, which is a natural generalization of matching complexes. The study of bounded degree complexes was initiated in a paper of Reiner and Roberts \cite{RR00}. Jonsson \cite{Jon08} further studied these complexes and derived connectivity bounds for BD$^{\vec{\lambda}}(K_n)$.  For more on these complexes, interested reader is referred to \cite{Jon08, Wach03}. 

In this article, we determine the homotopy type of bounded degree complexes of forests and prove \Cref{conj:vegas conjecture} as a special case. The main result of this article is the following.

\begin{theorem}\label{theorem:BD of forests main result}
Let $F$ be a forest on $n$ vertices and $\vec{\lambda}=(\lambda_1,\dots,\lambda_n)$ be any sequence of non-negative integers. Then, $\mathrm{BD}^{\vec{\lambda}}(F)$ is either contractible or homotopy equivalent to a wedge of spheres.
\end{theorem}

This article is organized as follows: In Section $2$, we present some definitions and results which are crucial to this article. In Section $3$, we prove \Cref{theorem:BD of forests main result}. In Section $4$, we determine the explicit homotopy type of the bounded degree complexes of those caterpillar graphs in which every non-leaf vertex is adjacent to at least one leaf vertex.

\section{Preliminaries}
A {\itshape graph} is an ordered pair $G=(V,E)$ where $V$ is called the set of vertices and $E \subseteq V \times V$, the set of unordered edges of $G$. The vertices $v_1, v_2 \in V$ are said to be adjacent, if $(v_1,v_2)\in E$. The number of vertices adjacent to a vertex $v$ is called {\itshape degree} of $v$, denoted deg$(v)$. If deg$(v)=1$, then $v$ is called a {\itshape leaf} vertex. A vertex $v$ is said to be {\itshape adjacent} to an edge $e$, if $v$ is an end point of $e$, {\itshape i.e.}, $e=(v, w)$. Two graphs $G$ and $H$ are called {\itshape isomorphic}, denoted $G\cong H$, if there exists a bijection, $f : V(G) \to V(H)$ such that $(v,w) \in E(G)$ if and only if $(f(v),f(w)) \in E(H).$ 

A graph $H$ with $V(H) \subseteq V(G)$ and $E(H) \subseteq E(G)$ is called a {\it subgraph} of the graph $G$. For a nonempty subset $H$ of $E(G)$, the induced subgraph $G[H]$, is the subgraph of $G$ with edges $E(G[H]) = H$ and $V(G[H]) = \{a \in  V(G)  :  a \text{ is adjacent to $e$ for some }e \in H\}$. For a nonempty subset $U$ of $V(G)$, the induced subgraph $G[U]$, is the subgraph of $G$ with vertices $V(G[U]) = U$ and $E(G[U]) = \{(a, b) \in E(G) \ | \ a, b \in U\}$.  

For $n \geq 1$, the {\itshape path graph} of length $n$, denoted $P_n$, is a graph with vertex set $V(P_n) = \{1, \ldots, n\}$ and  edge set $E(P_n) = \{(i,i+1) \ |  \ 1 \leq i \leq n-1\}$. A {\itshape tree} is a graph in which any two vertices are connected by exactly one path. The following graphs are a special class of trees.

\begin{definition}\label{def:caterpillar graph}
\normalfont A {\itshape caterpillar graph} is a tree in which every vertex is on a central path or only one edge away from the path (see \Cref{fig:caterpillar graph} for examples). 
\end{definition}

A caterpillar graph of length $n$ is denoted by $G_n(m_1,\dots,m_n)$, where $n$ represents the length of the central path and $m_i$ denote the number of leaves adjacent to the $i^{\mathrm{th}}$ vertex of the central path.

\begin{definition}\label{def:perfect caterpillar graphs}
\normalfont A caterpillar graph $G_n(m_1,\dots,m_n)$ is called {\itshape perfect} if $m_1=\dots=m_n=m$ and is denoted by $G_n(m)$.
\end{definition}

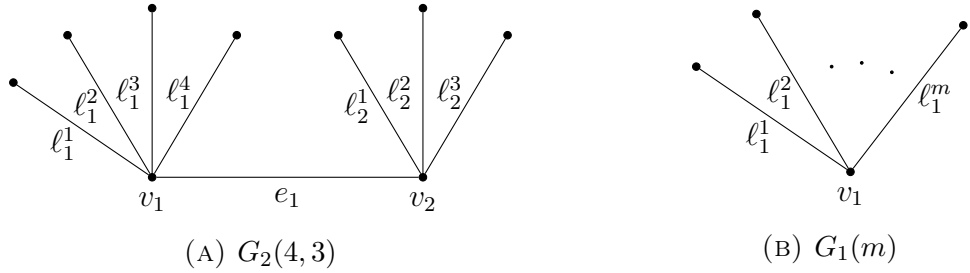
\begin{figure}[H]
	\begin{subfigure}[]{0.55 \textwidth}
		\centering
		\begin{tikzpicture}
 [scale=0.45, vertices/.style={draw, fill=black, circle, inner sep=1.0pt}]
        \node[vertices, label=below:{$v_1$}] (v1) at (0,0)  {};
		\node[vertices, label=below:{$v_2$}] (v2) at (8,0)  {};
		\node[vertices] (l11) at (-4.1,2.8)  {};
		\node[vertices] (l12) at (-2.5,4.2)  {};
		\node[vertices] (l13) at (0,5)  {};
		\node[vertices] (l14) at (2.5,4.2)  {};
		\node[vertices] (l21) at (5.5,4.2)  {};
		\node[vertices] (l22) at (8,5)  {};
		\node[vertices] (l23) at (10.5,4.2)  {};
		
\foreach \to/\from in {v1/v2}
%\draw [-] (\to)--(\from);
\path (v1) edge node[pos=0.5,below] {$e_1$} (v2);
\path (v1) edge node[pos=0.65,below] {$\ell_1^1$} (l11);
\path (v1) edge node[pos=0.5,left] {$\ell_1^2$} (l12);
\path (v1) edge node[pos=0.5,left] {$\ell_1^3$} (l13);
\path (v1) edge node[pos=0.6,left] {$\ell_1^4$} (l14);
\path (v2) edge node[pos=0.5,left] {$\ell_2^1$} (l21);
\path (v2) edge node[pos=0.5,left] {$\ell_2^2$} (l22);
\path (v2) edge node[pos=0.6,left] {$\ell_2^3$} (l23);
\end{tikzpicture}\caption{$G_2(4,3)$}\label{fig:G243}
	\end{subfigure}
	\begin{subfigure}[]{0.35 \textwidth}
		\centering
	\begin{tikzpicture}
 [scale=0.5, vertices/.style={draw, fill=black, circle, inner sep=1.0pt}]
        \node[vertices, label=below:{$v_1$}] (v1) at (0,0)  {};
		\node[vertices] (l11) at (-4.1,2.8)  {};
		\node[vertices] (l12) at (-2.5,4.2)  {};
		\node[vertices] (l1m) at (3.0,3.9)  {};
		\node[vertices,inner sep=0.3pt] (d1) at (-0.5,2.8)  {};
		\node[vertices,inner sep=0.3pt] (d2) at (0.3,2.9)  {};
		\node[vertices,inner sep=0.3pt] (d3) at (1.1,2.65)  {};
		
\foreach \to/\from in {v1/l11}
%\draw [-] (\to)--(\from);
\path (v1) edge node[pos=0.6,below] {$\ell_1^1$} (l11);
\path (v1) edge node[pos=0.5,left] {$\ell_1^2$} (l12);
%\path (v1) edge node[pos=0.5,left] {$l_1^3$} (l13);
\path (v1) edge node[pos=0.5,right] {$\ell_1^m$} (l1m);
\end{tikzpicture}\caption{$G_1(m)$}\label{fig:G1m}
	\end{subfigure}
	\caption{Caterpillar graphs} \label{fig:caterpillar graph}
\end{figure}

An {\itshape (abstract) simplicial complex} $K$ is a collection of finite sets such that if $\sigma \in K$ and $\tau \subseteq \sigma$, then $\tau \in K$. The elements  of $K$ are called {\itshape simplices} of $K$.  If $\sigma \in K$ and $|\sigma |=k+1$, then $\sigma$ is said to be {\it $k$-dimensional}. Further, if $\si \in K$ and $\tau \subseteq \si$ then $\tau$ is called a {\itshape face} of $\si$ and if $\tau \neq \si$ then $\tau$ is called a {\itshape proper face} of $\si$. The set of $0$-dimensional simplices of $K$ is denoted by $V(K)$, and its elements are called {\it vertices} of $K$. A {\it subcomplex} of a simplicial complex $K$ is a simplicial complex whose simplices are contained in $K$. For $k\geq 0$, the {\itshape $k$-skeleton} of a simplicial complex $K$ is the collection of all those simplices of $K$ whose dimension is at most $k$.

For a simplex $\sigma \in K$, define
\begin{align*}
(K:\sigma) & := \{\tau \in K : \sigma \cap \tau = \emptyset, ~\sigma \cup \tau \in K\}, \text{ and} \\
(K,\sigma) & := \{\tau \in K : \sigma \nsubseteq \tau \}.
\end{align*}

The simplicial complexes $(K:\sigma)$ and $(K,\sigma)$ are called {\itshape link} of $\sigma$ in $K$ and {\itshape face-deletion} of $\sigma$ respectively. The {\itshape join} of two simplicial complexes $K_1$ and $K_2$, denoted $K_1 \ast K_2$, is a simplicial complex whose simplices are disjoint union of simplices of $K_1$ and of $K_2$. The {\itshape cone} on $K$ with apex $a$, denoted $C_a(K)$, is defined as 
$$C_a(K) := K \ast \{\emptyset, a \}.$$ 

Note that, for any vertex $a\in V(K)$, we have
 \begin{equation}\label{eq:kasunionoflinkanddeletion}
 K= C_a(K:a) \cup (K,a) \text{ and } C_a(K:a) \cap (K,a)= (K:a).
 \end{equation}

For $a,b\notin V(K)$, the {\itshape suspension} of $K$, denoted $\Sigma(K)$, is defined as $$\Sigma (K)  := K \ast \{\emptyset , a, b\}.$$ 
 
For a simplicial complex $K$, let ${\Sigma}^r(K)$ denote its $r$-fold suspension, where $r\geq 1$ is a natural number. For $r\geq 0$, if $\S^{r}$ denotes $r$-dimensional sphere, then there is a homotopy equivalence
 \begin{equation}\label{eq:join of space with spheres is suspension}  \S^{r} \ast K \simeq \Sigma^{r+1}(K).
\end{equation}

\begin{definition}[{\cite[Definition 3.2]{MT08}}]\label{def:grape}
A simplicial complex $K$ is a combinatorial grape if
\begin{enumerate}
\item there is $a\in V(K)$ such that $(K:a)$ is contained in a cone contained in $(K,a)$ and both $(K:a)$ and $(K,a)$ are combinatorial grapes, or
\item $K$ has at most one vertex.
\end{enumerate}
\end{definition}

If $(K:a)$ is containded in a cone contained in $(K,a)$ for some $a \in V(K)$, then from \Cref{eq:kasunionoflinkanddeletion} and \cite[Proposition 0.18]{H02} we get the following homotopy equivalence \begin{equation}\label{eq:kaswedgeoflink}
K \simeq (K,a) \bigvee \Sigma (K:a).
\end{equation}

In this article, whenever we say that a simplicial complex is a grape, we shall mean that it is a combinatorial grape. The notion of grape was introduced by Marietti and Testa in \cite{MT08}. They proved the following result.

\begin{proposition}[{\cite[Proposition 3.3]{MT08}}]\label{prop:grapeimpliessphere}
If a simplicial complex $K$ is a grape, then each connected component of $K$ is either contractible or homotopy equivalent to a wedge of spheres.
\end{proposition}

\begin{definition}
\normalfont For $k\geq 1$, a \emph{$k$-matching} of a graph $G$ is subset of edges $H\subseteq E(G)$ such that any vertex $v \in G[H]$ has degree at most $k$. The \emph{$k$-matching complex} of a graph $G$, denoted $M_k(G)$, is a simplicial complex whose vertices are the edges of $G$ and faces are given by $k$-matchings of $G$. 
\end{definition}
 
\begin{definition}
\normalfont Let $\vec{\lambda}=(\lambda_1,\dots,\lambda_n)$ be an arbitrary sequence of non-negative integers and $G$ be a graph with vertex set $\{v_1,\dots,v_n\}$. The \emph{bounded degree complex}, denoted $\text{BD}^{\vec{\lambda}}(G)$, is a simplicial complex whose vertices are the edges of $G$, and faces are the subsets $H\subseteq E(G)$ such that the degree of vertex $v_i$ in the induced subgraph $G[H]$ is at most $\lambda_i$, for each $i \in \{1,\dots,n\}$. 
\end{definition}

{\bf Example:}
\Cref{fig:example of BD} consists of a graph $G$ on $5$ vertices and BD$^{\vec{\lambda}}(G)$ for $\vec{\lambda}=(2,1,1,1,1)$. The complex BD$^{\vec{\lambda}}(G)$ consists of $3$ maximal simplices, namely $\{e_1,\ell_1^1\}, \{e_1,\ell_1^2\}$ and $\{\ell_1^1,\ell_1^2,\ell_2^1\}$.

\begin{figure}[H]
	\begin{subfigure}[]{0.45 \textwidth}
		\centering
		\begin{tikzpicture}
 [scale=0.4, vertices/.style={draw, fill=black, circle, inner sep=1.0pt}]
        \node[vertices, label=below:{$v_1$}] (v1) at (0,0)  {};
		\node[vertices, label=below:{$v_2$}] (v2) at (8,0)  {};
		\node[vertices, label=above:{$v_3$}] (l11) at (-2.5,3.2)  {};
		\node[vertices, label=above:{$v_4$}] (l12) at (0,4)  {};
		\node[vertices, label=above:{$v_5$}] (l21) at (7,3.7)  {};
		
\foreach \to/\from in {v1/v2}
%\draw [-] (\to)--(\from);
\path (v1) edge node[pos=0.5,below] {$e_1$} (v2);
\path (v1) edge node[pos=0.5,left] {$\ell_1^1$} (l11);
\path (v1) edge node[pos=0.5,left] {$\ell_1^2$} (l12);
\path (v2) edge node[pos=0.5,left] {$\ell_2^1$} (l21);
\end{tikzpicture}\caption{$G_2(2,1)$}\label{fig:G221}
	\end{subfigure}
	\begin{subfigure}[]{0.45 \textwidth}
		\centering
	\begin{tikzpicture}
 [scale=0.28, vertices/.style={draw, fill=black, circle, inner sep=0.5pt}]
\node[vertices, label=below:{$\ell_1^1$}] (a) at (4,-4) {};
\node[vertices, label=above:{$\ell_1^2$}] (b) at (4,4) {};
\node[vertices, label=left:{$e_1$}] (c) at (-3,0) {};
\node[vertices, label=right:{$\ell_2^1$}] (f) at (11,0) {};

\foreach \to/\from in {a/b,a/c,a/f,b/c,b/f}
\draw [-] (\to)--(\from);
%\filldraw[fill=gray!60, draw=black] (2,3)--(9, 6.5)--(15,3)--cycle;
\filldraw[fill=gray!60, draw=black] (4,-4)--(4,4)--(11,0)--cycle;
%\draw [dashed] (a)--(b);
\end{tikzpicture}\caption{$\mathrm{BD}^{(2,1,1,1,1)}(G_2(2,1))$}\label{fig:BD of G}
	\end{subfigure}
	\caption{} \label{fig:example of BD}
\end{figure}

When $\vec{\lambda} = (\underbrace{k,\dots,k}_{n-\text{times}})$, we write BD$^k(G):=$ BD$^{\vec{\lambda}}(G)$. Clearly, BD$^k(G)$ is the $k$-matching complex $M_k(G)$.

The following results will be used repeatedly in this article. 

\begin{proposition}\label{prop:BD of disjoint union}
Let $G$ and $H$ be two graphs such that $V(G)=\{v_1,\dots,v_m\}$ and $V(H)=\{v_{m+1},\dots,v_{m+n}\}$. Then,
$$\mathrm{BD}^{(\lambda_1,\dots,\lambda_{m+n})}(G \sqcup H) = \mathrm{BD}^{(\lambda_1,\dots,\lambda_{m})}(G) \ast \mathrm{BD}^{(\lambda_{m+1},\dots,\lambda_{m+n})}(H),$$
where $\ast$ denotes the join operation.
\end{proposition} 

Proof of \Cref{prop:BD of disjoint union} follows directly from the definition of bounded degree complex.

\begin{lemma}[{\cite[Lemma 2.5]{BW95}}]\label{lemma:join of spheres} Suppose that $K, K_1$ and $K_2$ are finite simplicial complexes.
\begin{enumerate}
    \item If $K_1$ and $K_2$ both have the homotopy type of a wedge of spheres, then so does $K_1\ast K_2$.
    \item If $K$ has the homotopy type of a wedge of spheres, then so does $\Sigma(K)$.
    \item $\Big{(}\bigvee\limits_{i} \S^{a_i}\Big{)} \ast \Big{(}\bigvee\limits_{j} \S^{b_j}\Big{)} \simeq \bigvee\limits_{i,j} \S^{a_i+b_j+1}.$
\end{enumerate}
\end{lemma}

Henceforth, $[n]$ will denote the set $\{1,\dots,n\}$.

\section{Proof of \texorpdfstring{\Cref{theorem:BD of forests main result}}{sthm}}

The aim of this section is to determine the homotopy type of the bounded degree complexes of forests. To do so, we first prove a general result.

\begin{lemma}\label{lemma:homotopy of BD of tree}
Let $F$ be a forest on $n$ vertices and let $\vec{\lambda}=(\lambda_1,\dots,\lambda_n)$ be a labeling of the vertices of $F$ by non-negative integers. Then, the bounded degree complex $\mathrm{BD}^{\vec{\lambda}}(F)$ is a grape.
\end{lemma}
\begin{proof}
We proceed by induction on the number $m$ of edges of $F$. If $F$ has no edges, then the result is clear.

Suppose that $F$ has $m\geq 1$ edges and that all bounded degree complexes on
forests with at most $m-1$ edges are grapes. Since isolated vertices of $F$ do not affect $\mathrm{BD}^{\vec{\lambda}}(F)$, we can assume that $F$ does not have any isolated vertex. If $\lambda_i=0$ for some $i \in [n]$, then again the result is clear, as $\mathrm{BD}^{\vec{\lambda}}(F)$ is
either empty or equal to the bounded degree complex on a forest with fewer edges. If $F$ has an isolated edge $e$, then again the result is clear as the complex $\mathrm{BD}^{\vec{\lambda}}(F)$ is a cone with apex $e$. Otherwise, suppose that $\ell, v, w$ are distinct, consecutive vertices of $F$ ({\itshape i.e.} $\ell, v, w$ is a path of length $3$) and that $\ell$ is a leaf. Denote by $F^\prime$ the forest obtained from $F$ by removing the edge $\{v,w\}$ and by $\vec{\lambda^\prime}$ the labeling of $F^\prime$ given by
\begin{equation*}
\vec{\lambda^\prime}(u)= 
\begin{cases}
\vec{\lambda}(u), & \mathrm{if~} u \notin \{v,w\},\\
\vec{\lambda}(u)-1, & \mathrm{if~} u \in \{v,w\}.
\end{cases}
\end{equation*}

It is easy to observe that
\begin{align*}
\big{(}\mathrm{BD}^{\vec{\lambda}}(F) , \{v,w\} \big{)} & = \mathrm{BD}^{\vec{\lambda}}(F^\prime), \mathrm{~~ and }\\
\big{(}\mathrm{BD}^{\vec{\lambda}}(F) : \{v,w\} \big{)} & = \mathrm{BD}^{\vec{\lambda^\prime}}(F^\prime).
\end{align*}

From induction, both the simplicial complexes $\mathrm{BD}^{\vec{\lambda}}(F^\prime)$ and $\mathrm{BD}^{\vec{\lambda^\prime}}(F^\prime)$ are grapes. Moreover, the cone on $\mathrm{BD}^{\vec{\lambda^\prime}}(F^\prime)$ with apex $\{\ell,v\}$ is entirely contained in $\mathrm{BD}^{\vec{\lambda}}(F^\prime)$. Therefore, we conclude that $\mathrm{BD}^{\vec{\lambda}}(F)$ is a grape.
\end{proof}

The following general result implies that the bounded degree complex of a forest can have at most one connected component that is not a single point.

\begin{lemma}\label{lemma:connectedcomponent}
 Let $G = (V, E)$ be a simple graph and let $\vec{\lambda} : V \rightarrow \mathbb{Z}_{\geq 0}$ be a labelling of the vertices of $G$. If the bounded degree complex $\mathrm{BD}^{\vec{\lambda}}(G)$ contains two connected components with more than one vertex, then $G$ contains a cycle of length $4$.
\end{lemma}

\begin{proof} Suppose that $e_1 , e_2 , f_1 , f_2 \in E$ are edges of the graph $G$ such that
\begin{itemize}
\item $e_1 , e_2$ are adjacent as vertices of $\mathrm{BD}^{\vec{\lambda}}(G)$;
\item $f_1 , f_2$ are adjacent as vertices of $\mathrm{BD}^{\vec{\lambda}}(G)$;
\item $e_1$ and $f_1$ are in distinct connected components of $\mathrm{BD}^{\vec{\lambda}}(G)$.
\end{itemize}
We deduce that the labels of all the endpoints of $e_1 , e_2 , f_1 , f_2$ are at least $1$. Moreover, for $i, j \in \{1, 2\}$, the intersection $e_i \cap f_j$ is not empty, since otherwise we violate the condition that $e_1$ and $f_1$ belong to different connected components of $\mathrm{BD}^{\vec{\lambda}}(G)$. If $e_1 , e_2$ have a common vertex $v$, then, necessarily, $\vec{\lambda}(v) \geq 2$. Our assumptions then force $f_1$ and $f_2$ to not contain $v$, but still intersect non-trivially both edges $e_1 , e_2$ . This is impossible, since the graph $G$ is simple. By symmetry, also $f_1 , f_2$ do not have a vertex in common. We conclude that $\{e_1 , e_2 \}$ and $\{f_1 , f_2 \}$ are both matchings of G with the same support. Thus, $e_1 , f_1 , e_2 , f_2$ is a square in $G$.
\end{proof}

Combining \Cref{prop:grapeimpliessphere}, \Cref{lemma:homotopy of BD of tree} and \Cref{lemma:connectedcomponent}, we get that the bounded degree complexes of forests are either contractible or homotopy equivalent to a wedge of spheres. This completes the proof of \Cref{theorem:BD of forests main result}.
\vspace*{0.2cm}

\noindent{\bf Note:} The above proof of \Cref{theorem:BD of forests main result} is due to the anonymous referee. This can also be proved using \cite[Lemma 10.4(ii)]{B95} and \Cref{lemma:join of spheres}. The earlier version of the manuscript contained a proof along those lines. 

\vspace*{0.2cm}
For an arbitrary forest $F$, it is difficult to obtain a closed form formula for the homotopy type of $\mathrm{BD}^{\vec{\lambda}}(F)$. Here, we provide a simple recursive formula that gives the number of spheres of each dimension occurring in the homotopy type of $\mathrm{BD}^{\vec{ \lambda }}(F)$. 

Denote by $(F,\vec{\lambda})_d$ the number of $d$-dimensional spheres occurring in the homotopy type of $\mathrm{BD}^{\vec{ \lambda }}(F)$. For an edge $e$ of $F$, denote by $F\setminus e$ the forest obtained from $F$ by removing edge $e$. Let $\vec{\lambda}_e$ be the labelling of the vertices of $F\setminus e$ given by
\begin{equation*}
\vec{\lambda}_e(u)= 
\begin{cases}
\vec{\lambda}(u), & \mathrm{if~} u \notin e,\\
\vec{\lambda}(u)-1, & \mathrm{if~} u \in e.
\end{cases}
\end{equation*}

\begin{proposition}\label{prop:recurranceforforest}
Let $F$ be a forest, let $\vec{\lambda}$ be a labelling of the vertices of $F$ by non-negative integers. Suppose that $e$ is an edge of $F$ and that there is a leaf vertex $u \notin e$ adjacent to $e$. Then, for all integers $d$, the equality
\begin{equation}
(F,\vec{\lambda})_d = (F\setminus e,\vec{\lambda})_d + (F\setminus e, \vec{\lambda}_e )_{d-1}
\end{equation}
holds.
\end{proposition}
\begin{proof}
From the proof of \Cref{lemma:homotopy of BD of tree} and \Cref{eq:kaswedgeoflink} we get 
 \begin{equation*}
 \begin{split}
\mathrm{BD}^{\vec{\lambda}}(F) & \simeq \big{(}\mathrm{BD}^{\vec{\lambda}}(F), e\big{)} \bigvee \Sigma \big{(} \mathrm{BD}^{\vec{\lambda}}(F) : e \big{)} \\
& \simeq \mathrm{BD}^{\vec{\lambda}}(F\setminus e)\bigvee \Sigma\big{(} \mathrm{BD}^{\vec{\lambda}_e}(F\setminus e) \big{)}.
\end{split}
 \end{equation*}
 This completes the proof of \Cref{prop:recurranceforforest}
\end{proof}

\Cref{prop:recurranceforforest} gives a recursion for computing the integers $(F,\vec{\lambda})_d$ in terms of similar numbers for subforests of $F$ with fewer edges. The recursion terminates on the subforests $F_0 \subseteq F$ that are disjoint unions of paths with at most two vertices, where the initial conditions are
\begin{equation}
(F_0,\vec{\lambda})_d =
\begin{cases}
\vspace*{0.2cm}
1, & \parbox[t]{.45\columnwidth}{if $d = -1$, $\vec{\lambda}$ is non-negative and at least one vertex of every edge of $F_0$ has label $0;$} \\ 
 0, & \mathrm{otherwise}.
\end{cases}
\end{equation}

\section{Bounded degree complexes of caterpillar graphs}

Let $\vec{\lambda}=(\lambda_1,\dots,\lambda_n)$ and $\vec{m}=(m_1,\dots,m_n)$ be two sequences of non-negative integers. The vertices and edges of the central path of graph $G_n(\vec{m})$ will be denoted by $\{v_1,\dots,v_n\}$ and $\{e_1,\dots,e_{n-1}\}$ respectively as shown in \Cref{fig:example of general caterpillar graph}.

\begin{figure}[H]
		\centering
		\begin{tikzpicture}
 [scale=0.45, vertices/.style={draw, fill=black, circle, inner sep=1.0pt}]
        \node[vertices, label=below:{$v_1$}] (v1) at (0,0)  {};
		\node[vertices, label=below:{$v_2$}] (v2) at (7.5,0)  {};
		\node[vertices, label=below:{$v_{n-2}$}] (vn2) at (13.5,0)  {};
		\node[vertices, label=below:{$v_{n-1}$}] (vn1) at (21,0)  {};
		\node[vertices, label=below:{$v_n$}] (vn) at (28,0)  {};
		\node[vertices,inner sep=0.2pt] (d1) at (9.7,0)  {};
		\node[vertices,inner sep=0.2pt] (d2) at (10.5,0)  {};
		\node[vertices,inner sep=0.2pt] (d3) at (11.3,0)  {};
		\node[vertices] (l11) at (-4.1,2.8)  {};
		\node[vertices] (l12) at (-2.5,4.2)  {};
		\node[vertices] (l1m) at (1.5,4.2)  {};
		\node[vertices,inner sep=0.2pt] (d11) at (-0.9,2.8)  {};
		\node[vertices,inner sep=0.2pt] (d12) at (-0.3,2.9)  {};
		\node[vertices,inner sep=0.2pt] (d1m) at (0.3,2.82)  {};
		\node[vertices] (l21) at (5.0,4.2)  {};
		\node[vertices] (l2m) at (9.0,4.2)  {};
		\node[vertices,inner sep=0.2pt] (d21) at (6.6,2.8)  {};
		\node[vertices,inner sep=0.2pt] (d22) at (7.2,2.9)  {};
		\node[vertices,inner sep=0.2pt] (d2m) at (7.8,2.82)  {};
		\node[vertices,inner sep=0.2pt] (dn21) at (12.6,2.8)  {};
		\node[vertices,inner sep=0.2pt] (dn22) at (13.2,2.8)  {};
		\node[vertices,inner sep=0.2pt] (dn2m) at (13.8,2.8)  {};
		\node[vertices] (ln11) at (18.5,4.2)  {};
		\node[vertices] (ln1m) at (22.5,4.2)  {};
		\node[vertices,inner sep=0.2pt] (dn11) at (20.1,2.8)  {};
		\node[vertices,inner sep=0.2pt] (dn12) at (20.7,2.9)  {};
		\node[vertices,inner sep=0.2pt] (dn1m) at (21.3,2.82)  {};
		\node[vertices] (ln1) at (26.0,4.2)  {};
		\node[vertices] (lnm) at (30,4.2)  {};
		\node[vertices,inner sep=0.2pt] (dn1) at (27.6,2.8)  {};
		\node[vertices,inner sep=0.2pt] (dn2) at (28.2,2.9)  {};
		\node[vertices,inner sep=0.2pt] (dnm) at (28.8,2.82)  {};
		
\foreach \to/\from in {v1/v2}
%\draw [-] (\to)--(\from);
\path (v1) edge node[pos=0.5,below] {$e_1$} (v2);
\path (vn2) edge node[pos=0.5,below] {$e_{n-2}$} (vn1);
\path (vn1) edge node[pos=0.5,below] {$e_{n-1}$} (vn);
\path (v1) edge node[pos=0.6,below] {$\ell_1^1$} (l11);
\path (v1) edge node[pos=0.5,left] {$\ell_1^2$} (l12);
\path (v1) edge node[pos=0.5,right] {$\ell_1^{m_1}$} (l1m);
\path (v2) edge node[pos=0.5,left] {$\ell_2^{1}$} (l21);
\path (v2) edge node[pos=0.5,right] {$\ell_2^{m_2}$} (l2m);
\path (vn1) edge node[pos=0.5,left] {$\ell_{n-1}^1$} (ln11);
\path (vn1) edge node[pos=0.5,right] {$\ell_{n-1}^{m_{n-1}}$} (ln1m);
\path (vn) edge node[pos=0.5,left] {$\ell_{n}^1$} (ln1);
\path (vn) edge node[pos=0.5,right] {$\ell_{n}^{m_{n}}$} (lnm);
%\path (v2) edge node[pos=0.5,left] {$\ell_2^1$} (l21);
\end{tikzpicture}
	\caption{$G_n(m_1,\dots,m_n)$} \label{fig:example of general caterpillar graph}
\end{figure}
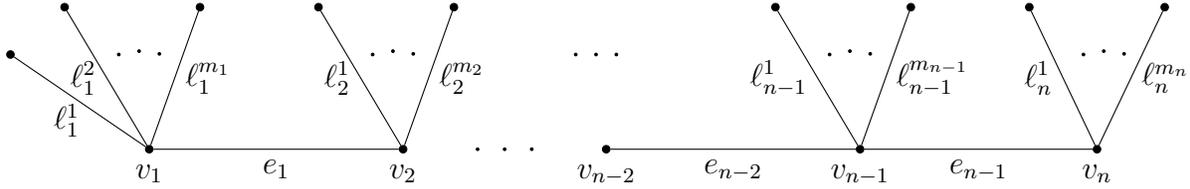

Hereafter, BD$^{\vec{\lambda}}(G_n(\vec{m}))$ will denote the bounded degree complex whose faces are subgraphs of $G_n(\vec{m})$ in which the degree of central path vertex $v_i$ is at most $\lambda_i$ for each $i \in [n]$ and degree of leaf vertices is at most $1$, {\itshape i.e.}, BD$^{\vec{\lambda}}(G_n(\vec{m}))=$ BD$^{\vec{\mu}}(G_n(\vec{m}))$, where $\vec{\mu}=(\lambda_1,\dots,\lambda_n,\underbrace{1,1\dots,1,1}_{(\sum\limits_{i=1}^{n} m_i)- \text{times}})$.

Clearly, if $\lambda_i=0$ for all $i \in [n]$ then BD$^{\vec{\lambda}}(G_n(\vec{m}))$ is the simplicial complex $\{\emptyset\}$. Therefore, we assume that $\vec{\lambda}\neq \vec{0}=(0,\dots,0)$. Since all caterpillar graphs are trees, the following result is an immediate corollary of \Cref{theorem:BD of forests main result}.

\begin{corollary}\label{theorem:caterpillarhomotopy}
Let $\vec{\lambda}=(\lambda_1,\dots,\lambda_n)$ and $\vec{m}=(m_1,\dots,m_n)$ be two sequences of non-negative integers. Then, the bounded degree complex $\mathrm{BD}^{\vec{\lambda}} (G_n(\vec{m}))$ is either contractible or homotopy equivalent to a wedge of spheres.
\end{corollary}

Recall that $k$-matching complex $M_k(G_n(\vec{m}))$ is same as the bounded degree complex BD$^{k}(G_n(\vec{m}))$. Therefore, \Cref{theorem:caterpillarhomotopy} implies \Cref{conj:vegas conjecture}.

A {\itshape cycle graph} on $n$ vertices, denoted $C_n$, is a graph with vertex set $[n]$ and edge set $E(C_n)= \{(i,i+1) : i \in [n-1]\} \cup \{(1,n)\}$. In \cite[Proposition 5.2]{Koz99}, Kozlov showed that the complex $M_1(C_n)$ is homotopy equivalent to a wedge of spheres. When $\vec{\lambda} \neq (1,\dots,1)$, we can use \Cref{theorem:caterpillarhomotopy} to determine the homotopy type of BD$^{\vec{\lambda}}(C_n)$. 

\begin{corollary}\label{corollary:BD of cycle graph}
For any $\vec{\lambda}=(\lambda_1,\dots,\lambda_n)$, $\mathrm{BD}^{\vec{\lambda}}(C_n)$ is either contractible or homotopy equivalent to a wedge of spheres.
\end{corollary} 
\begin{proof}
 For $\vec{\lambda}=(1,\dots,1)$, we get the result from \cite[Proposition 5.2]{Koz99}. Now, let $\vec{\lambda}\neq (1,\dots,1)$. Since the graph is cycle graph, without loss of generality, we can assume that $\lambda_n\neq 1$.
 
 If $\lambda_n=0$, then $\mathrm{BD}^{\vec{\lambda}}(C_n) = \mathrm{BD}^{(\lambda_1,\dots,\lambda_{n-1})}(C_n - \{n\}) = \mathrm{BD}^{(\lambda_1,\dots,\lambda_{n-1})}(P_{n-1})$. If $\lambda_n\geq 2$, then $\mathrm{BD}^{\vec{\lambda}}(C_n) = \mathrm{BD}^{(1,\lambda_1,\dots,\lambda_{n}-1)}(P_{n+1})$. In both the cases, result follows from \Cref{theorem:caterpillarhomotopy} (since $G_n(\vec{0}) \cong P_n$).
\end{proof}

In \cite[Theorem 5.11]{MJMV19}, authors determined the explicit homotopy type of $M_1(G_n(\vec{m}))$ where $m_i \geq 1$ for each $i \in [n]$. We generalize their result by determining the explicit homotopy type of $\mathrm{BD}^{\vec{\lambda}}(G_n(\vec{m}))$ where $m_i \geq 1$ for each $i\in [n]$. We first introduce a few notations. Denote by $P_n$ the central path of the caterpillar graph $G_n(\vec{m})$. For a subset $T \subseteq E(P_n)$ of the edges of the path $P_n$, and an index $i \in [n]$, denote by $T_i$ the degree of the $i^{\text{th}}$ vertex of the path $P_n$ in the
subgraph induced by $T$. Thus, for each $i$, the value of $T_i$ is $0, 1,$ or $2$. Also, a wedge over $0$ denotes a contractible space.

\begin{theorem}\label{lemma:exacthomotopyofcaterpillar}
Let $G_n(\vec{m})$ be a caterpillar graph such that every vertex in its central path is adjacent to at least one leaf, {\itshape i.e.}, $m_i \geq 1$ for each $i \in [n]$. Let $\vec{\lambda} = (\lambda_1, \dots , \lambda_n)$ and $|\vec{\lambda}|= \sum\limits_{i=1}^{n} \lambda_i$. Then, 
$$\mathrm{BD}^{\vec{\lambda}}(G_n(\vec{m})) \simeq \bigvee\limits_{T \subseteq E(P_n)} \left(  \bigvee\limits_{\prod_i \binom{m_i-1}{\lambda_i-T_i}} \S^{|\vec{\lambda}|- \# T-1} \right).$$
\end{theorem}
\begin{proof}
For $i \in [n-1]$, let $e_i$ be the $i^{\text{th}}$ edge in the central
path of the caterpillar graph $G_n(\vec{m})$ (see \Cref{fig:example of general caterpillar graph}). Since $\mathrm{BD}^{\vec{\lambda}}(G_n(\vec{m}))$ is a grape (\Cref{lemma:homotopy of BD of tree}) and each vertex of the central path of $G_n(\vec{m})$ is adjacent to a leaf, we have the following homotopy equivalence
$$\mathrm{BD}^{\vec{\lambda}}(G_n(\vec{m}))\simeq \big{(}\mathrm{BD}^{\vec{\lambda}}(G_n(\vec{m})), e_i\big{)} \bigvee \Sigma\big{(}\mathrm{BD}^{\vec{\lambda}}(G_n(\vec{m})): e_i\big{)} \mathrm{~for~each~} i \in [n-1].$$

Moreover, for each $i \in [n-1]$, it is easy to see that the face deletion  
$$\big{(} \mathrm{BD}^{\vec{\lambda}}(G_n(\vec{m})) , e_{i} \big{)}$$
is isomorphic to 
 $$\mathrm{BD}^{(\lambda_1,\dots,\lambda_{i})}(G_{i}(m_1,\dots,m_{i})) \ast \mathrm{BD}^{(\lambda_{i+1},\dots,\lambda_n)}(G_{n-i}(m_{i+1},\dots,m_n))$$
and the link 
$$\big{(} \mathrm{BD}^{\vec{\lambda}}(G_n(\vec{m})) : e_{i} \big{)} $$ is isomorphic to $$\mathrm{BD}^{(\lambda_1,\dots,\lambda_{i-1},\lambda_{i}-1)}(G_{i}(m_1,\dots,m_{i})) \ast \mathrm{BD}^{(\lambda_{i+1}-1,\lambda_{i+2},\dots,\lambda_n)}(G_{n-i}(m_{i+1},\dots,m_n)).$$

Combining these decompositions and isomorphisms, we can successively eliminate the edges in the central path of $G_n(\vec{m})$, until none is left. By doing so, we get the following homotopy equivalence
\begin{equation}\label{eq:computation}
\mathrm{BD}^{\vec{\lambda}}(G_n(\vec{m})) \simeq  \bigvee\limits_{T \subseteq E(P_n)}\bigg{(}\mathop{\ast}_{i=1}^{n} \Sigma^{T,i}\Big{(}\mathrm{BD}^{(\lambda_i-T_i)}(G_1(m_i))\Big{)} \bigg{)},
\end{equation}
where, for a simplicial complex $K$, $T \subseteq E(P_n)$ and $i\in [n]$, the space $\Sigma^{T,i}(K)$ is 
\begin{equation*}
\Sigma^{T,i}(K)=
\begin{cases}
\Sigma(K), & \mathrm{if~} i>1, e_{i-1} \in T, \\
K, & \mathrm{otherwise}. 
\end{cases}
\end{equation*}

Also, it is easy to see that, for $k,r \geq 1$, $\mathrm{BD}^k(G_1(r))$ is $(k-1)$-skeleton of a simplex of dimension $(r-1)$. Therefore, 
\begin{equation}\label{eq:BD of star graph}
\mathrm{BD}^k(G_1(r))\simeq 
\begin{cases}
\bigvee\limits_{\binom{r-1}{k}} \S^{k-1}, & \mathrm{if}~ k< r,\\
\{\mathrm{point}\}, & \mathrm{if}~ k\geq r.
\end{cases}
\end{equation}

Thus, using \Cref{lemma:join of spheres} and \Cref{eq:BD of star graph}, we get $$\Sigma^{T,i}\Big{(}\mathrm{BD}^{(\lambda_i-T_i)}(G_1(m_i))\Big{)} \simeq \left( \bigvee\limits_{\binom{m_i-1}{\lambda_i-T_i}} \S^{\lambda_i - T_i-1+\delta_{T,i}}\right),$$
where $\delta_{T,i} \in \{0,1\}$ is the characteristic function: 
\begin{equation*}
\delta_{T,i}=
\begin{cases}
1, & \mathrm{if~} i>1, e_{i-1} \in T, \\
0, & \mathrm{otherwise}. 
\end{cases}
\end{equation*}

Observe that, for any $T\subseteq E(P_n)$, $\sum\limits_{i=1}^{n} T_i =2 \#T$ and $\sum\limits_{i=1}^{n} \delta_{T,i} = \#T$.

Finally, distributing the join over the wedge in \Cref{eq:computation} and using \Cref{lemma:join of spheres} repeatedly, we obtain the equivalence 
$$\mathrm{BD}^{\vec{\lambda}}(G_n(\vec{m})) \simeq \bigvee\limits_{T \subseteq E(P_n)} \left(  \bigvee\limits_{\prod_i \binom{m_i-1}{\lambda_i-T_i}} \S^{|\vec{\lambda}|- \# T-1} \right).$$

This completes the proof of \Cref{lemma:exacthomotopyofcaterpillar}.
\end{proof}

\section{Acknowledgement}
The author would like to thank Priyavrat Deshpande for many helpful comments and suggestions. The author would also like to thank the anonymous referee for suggesting the concept of combinatorial grape, which helped a lot in improving this article. In particular, the author is grateful for the proof of \Cref{lemma:homotopy of BD of tree} which also lead to the strengthening of \Cref{lemma:exacthomotopyofcaterpillar}. This work is partially funded by a grant from Infosys Foundation.

\bibliographystyle{alpha}

\end{document}